\documentclass[11pt]{article}
\usepackage[utf8]{inputenc}

\usepackage[english]{babel}
\usepackage{amsthm,amssymb,amsfonts,amsmath,graphicx,cite, color,mathrsfs,booktabs}

\usepackage{amsthm}
\usepackage[makeroom]{cancel}
\usepackage{caption}
\usepackage{hyperref}
\hypersetup{
    colorlinks=true,
    linkcolor=blue,
    filecolor=magenta,
    urlcolor=cyan,
}
\usepackage{amsmath, graphicx,amssymb, latexsym, amsfonts, amscd, amsthm,mathtools ,xypic, setspace, color, hyperref,caption,bbm}
\usepackage{enumerate, enumitem}
\usepackage[dvipsnames]{xcolor}

\usepackage{rotating} 

\usepackage[margin=3cm]{geometry}

\theoremstyle{definition}

\theoremstyle{plain}
\newtheorem{theorem}{Theorem}
\newtheorem{lemma}{Lemma}
\newtheorem{corollary}{Corollary}

\theoremstyle{remark}

\newtheorem{example}{Example}

\newcommand{\T}{\mathcal T}

\renewcommand{\SS}{{\mathscr S}}
\newcommand{\Diag}{\operatorname{Diag}}
\newcommand{\ip}[2]{\langle #1 , #2 \rangle}    

\newcommand{\R}{\mathbb{R}}
\newcommand{\matr}[1]{\begin{bmatrix} #1 \end{bmatrix}}    
\renewcommand{\H}{\mathcal H}
\newcommand{\dom}{{\mathrm{dom}}}
\DeclareMathOperator{\Image}{Im}

\newcommand{\graph}{{\mathrm{graph}}}
\newcommand{\ri}{{\mathrm{ri}}}
\renewcommand{\S}{{\mathfrak S}}
\def\transp{^{\text{\sf T}}}
\newcommand{\dist}{{\mathrm{dist}}}
\newcommand{\B}{\mathbb{B}}
\author{Javier F. Pe\~na\thanks{Tepper School of Business,
Carnegie Mellon University, USA, {\tt jfp@andrew.cmu.edu}} \and Juan C. Vera\thanks{Department of Econometrics and
Operations Research,
Tilburg University, The Netherlands, {\tt j.c.veralizcano@uvt.nl}}
\and Luis F. Zuluaga\thanks{Department of Industrial and Systems Engineering, Lehigh University, USA, {\tt
luis.zuluaga@lehigh.edu}}}
\title{Duality of Hoffman constants}
\begin{document}

\maketitle

\begin{abstract}

We show that a suitable Slater condition implies a {\em duality inequality} between the Hoffman constants of the following feasibility problems:
\[
\begin{array}{r}
Ax-b \in S\\
x \in R
\end{array} 
\qquad\text{ and }\qquad  
\begin{array}{r}
c-A\transp y \in R^*\\
y \in S^*,
\end{array}
\]
where $A\in \R^{m\times n}$, and $R\subseteq \R^n$ and $S\subseteq \R^m$ are {\em reference} polyhedral cones, with respective dual cones $R^*\subseteq \R^n$ and $S^*\subseteq \R^m$.
Our approach relies on an exact characterization of Hoffman constants and introduces a novel {\em Hoffman duality inequality} for polyhedral set-valued mappings. These two fundamental results also yield a striking identity between the Hoffman constants of {\em box-constrained} feasibility problems, which feature a similar primal-dual structure with a box and a linear subspace as reference sets. Additionally, we establish a surprising identity between the Hoffman constants of box-constrained feasibility problems and the chi condition measures for weighted least-squares problems.
\end{abstract}


\section{Introduction}

Several recent papers show a surge of interest in the central role that Hoffman bounds, and more general error bounds, play in mathematical optimization.   In particular, Hoffman constants play a central role in the modern development of first-order algorithms for linear programming~\cite{ApplHLL22,Hind23,LuY23} and in recent advances on the conditional gradient method~\cite{GutmP21,PenaR19,Zhao23}.  Error bounds are instrumental for many other results on the convergence of a variety of optimization algorithms~\cite{ApplHLL22,BeckS17,DaduNV20,GaoK20,LacoJ15,LeveL10,LuoT93,WangL14, xia2020globally,YeYZZ21}.

  The now vast literature on error bounds (see~\cite{AzeC02,burke1996,guler1995,Li93,MangS87,PenaVZ21,robinson1973,Zali03,zheng2004} and the many references therein) started with the seminal paper of Hoffman~\cite{Hoff52}, who showed that the distance from a point $v\in \R^m$ 
 to a nonempty polyhedron of the form $\{y\in \R^m: A\transp y \le c\}$ can be bounded above in terms of a constant that depends only on $A$ and the magnitude of the violation of the constraints $A\transp v \le c$.  

The Hoffman error bound extends to the following more general context.  Suppose $A\in \R^{m\times n}$, and $R\subseteq \R^n$ and $S\subseteq \R^m$ are {\em reference} polyhedra such as a polyhedral cone, a box, or a simplex.
Hoffman's bound implies that the distance from a point $u\in R$ to a nonempty polyhedron of the form $\{x\in R: Ax-b\in S\}$ can also be bounded in terms of a {\em Hoffman constant} that depends only on the triple $(A,R,S)$ and the magnitude of the violation of the constraints $Ax-b\in S$. 

These general Hoffman constants are a key component for establishing linear convergence in various first-order algorithmic schemes, as featured in recent papers~\cite{ApplHLL22,Hind23,LuY23}. A related development is the technique in~\cite{xia2020globally}, which relies on the characterization of Hoffman constants from~\cite{PenaVZ21} to construct tight integer programming formulations for complementarity constraints. In addition,~\cite{van2023bounding,pena2021linear} use new characterizations of Hoffman constants to develop novel analyses of linear convergence for other popular optimization algorithms, such as the Douglas-Rachford method and the closely related Alternating Direction Method of Multipliers, without the need for
strong convexity assumptions.

The renewed interest in Hoffman constants and the central role that duality plays in optimization have inspired us to investigate the duality properties of Hoffman constants.  More precisely, this paper is concerned with the following natural duality question concerning Hoffman constants:

\begin{quote}
Suppose $A\in \R^{m\times n}$ and $R\subseteq \R^n$ and $S\subseteq \R^m$ are  reference polyhedral cones  with dual cones $R^*\subseteq \R^n, \; S^*\subseteq \R^m$.  What is the relationship between the Hoffman constants of the feasibility problems
\[
\begin{array}{r}
Ax-b \in S\\
x \in R
\end{array}
\qquad 
\text{ and }
\qquad
\begin{array}{r}
c-A\transp y \in R^*\\
y \in S^*?
\end{array}
\]
\end{quote}
We show a strong relationship between these Hoffman constants.  More precisely, we show that a suitable {\em Slater} condition implies a  {\em duality inequality} between the Hoffman constants of the above two feasibility problems (Corollary~\ref{cor.onesided.dual}). The direction of this inequality depends on which Slater condition holds.  
The crux of our duality developments are a {\em Hoffman duality inequality} for solution mappings of general polyhedral feasibility problems (Theorem~\ref{thm.dual.ineq}) 
and an extension of our previous Hoffman constant characterization~\cite{PenaVZ21} (Theorem~\ref{thm.Hoffman.char}). 

The Slater condition plays a key role in the duality properties in Corollary~\ref{cor.onesided.dual}.  This may be counterintuitive since strong duality always holds for primal-dual convex polyhedral optimization problems.  This role of the Slater condition lies at the core of both Theorem~\ref{thm.Hoffman.char} and Theorem~\ref{thm.dual.ineq}, as it leads to a convenient characterization of the Hoffman constant (see identity~\eqref{eq.hoffman.special} in Theorem~\ref{thm.Hoffman.char}) and to a general Hoffman duality inequality (see~Theorem~\ref{thm.dual.ineq}).

The duality properties in Corollary~\ref{cor.onesided.dual} are inherently asymmetric. This is because the Slater condition can only hold for one of the feasibility problems, as illustrated by Examples~\ref{ex.primal} and~\ref{ex.dual}. However, Theorem~\ref{thm.Hoffman.char} and Theorem~\ref{thm.dual.ineq} lead to a striking symmetric {\em duality identity} for box-constrained feasibility problems: Let $\B \subseteq \R^n$ be a {\em box} and $L\subseteq \R^m$ be a {\em linear subspace} with orthogonal complement $L^\perp\subseteq \R^m$. In Theorem~\ref{thm.box.dual}  we show that the Hoffman constants of the following box-constrained feasibility problems are identical:
\[
\begin{array}{r}
Ax-b \in L\\
x \in \B
\end{array}
\qquad 
\text{ and }
\qquad
\begin{array}{r}
c-A\transp y \in \B\;\;\;\\
y \in L^\perp.
\end{array}
\]

We also show that if $\R^n$ and $\R^m$ are endowed with Euclidean norms, then for suitable $A$ and $L$, the Hoffman constants of these box-constrained feasibility problems match the {\em chi} condition measures $\chi(A)$ and $\overline\chi(A)$ (Theorem~\ref{thm.chi.hoffman}). The chi condition measures were introduced in~\cite{Olea90,Stew89}. They play a key role in weighted least squares~\cite{BobrV01,Fors96,ForsS01} and in the layered-step interior-point algorithm of Vavasis and Ye~\cite{MontT03,VavaY96}. The identities between Hoffman constants of box-constrained feasibility problems and chi condition measures (Theorem~\ref{thm.chi.hoffman}) are notable since these constants are defined differently and apply to distinct problems.

Similar to the chi condition measures, the {\em kappa circuit imbalance measure} was recently developed by Dadush, Natura, and V\'egh~\cite{DaduNV20} and by Dadush, Huiberts, Natura, and V\'egh~\cite{DaduHNV20,DaduHNV24}. The kappa measure plays an analogous role to that of the chi measures in~\cite{MontT03,VavaY96} but for more recent algorithmic developments for linear programming~\cite{DaduHNV20,DaduHNV24,EkbaNV22,Natu22}. We conjecture that, with an appropriate choice of norms, an identity like the one for the chi measures (Theorem~\ref{thm.chi.hoffman}) also holds. Specifically, this identity would equate the Hoffman constant to the kappa circuit imbalance measure in lieu of the chi measure. Indeed, key results in~\cite{DaduNV20} indicate a connection between Hoffman constants and the kappa circuit measure. A glance at~\cite[Corollary 3.2 and Corollary 3.3]{DaduNV20} suggests that the kappa measure serves as a proxy for a Hoffman constant. 
\medskip

Our duality results rely on a characterization of the Hoffman constant in~\cite{PenaVZ21}. This characterization is similar in spirit to those previously documented in~\cite{AzeC02,guler1995,KlatT95,Zali03,zheng2004}. Although these earlier results have been known for decades, to the best of our knowledge, there are no duality results of the kind we develop here.  Developing such results has been challenging due to the lack of exact characterizations of Hoffman constants for general feasibility problems of the form $Ax-b\in S,\; x\in R$ where $S$ and $R$ are reference polyhedral sets. 
Indeed, most of the existing characterizations of Hoffman constants, including those in~\cite{AzeC02,guler1995,KlatT95,Zali03,zheng2004}, consider only feasibility problems in {\em standard inequality-equality format} $Ax \le b, \, Cx = d$ or the more restricted  {\em standard inequality-only format} $Ax \le b$.  In principle, any feasibility problem of the form $Ax-b\in S,\; x\in R$ for reference polyhedra $S$ and $R$ can be recast in standard inequality-equality or inequality-only format.  Therefore, 
 {\em upper bounds} on its Hoffman constant readily follow from the Hoffman constant of its standard inequality-equality reformulation.  However, this approach yields only loose upper bounds, since recasting in inequality-equality format would inevitably introduce spurious dependencies on the descriptions of $R$ and $S$.  This shortcoming applies, for example, to~\cite[Prop. 3]{xia2020globally}.  Without a characterization of the exact Hoffman constants for general feasibility problems of the form $Ax-b\in S,\ x\in R$, it is impossible to state a meaningful comparison of Hoffman constants between different problem instances. As we formally state in Theorem~\ref{thm.Hoffman.char}, the more recent article~\cite{PenaVZ21} can be leveraged to {\em exactly} characterize Hoffman constants for the general class of feasibility problems that we consider. This is a key component of our duality developments.

\subsection{Overview of main results}\label{sec.overview}

We will now summarize our main results and introduce some of the notation and terminology used throughout this paper. Let \( A \in \mathbb{R}^{m \times n} \), and let \( R \subseteq \mathbb{R}^n \) and \( S \subseteq \mathbb{R}^m \) be some {\em reference polyhedra.} Consider the feasibility problem defined by the triple \( (A, R, S) \) and $b\in \R^m$: 
\begin{equation}
\label{eq.onesided.systems.general}
\begin{array}{r} Ax - b \in S \\ x \in R. \end{array} 
\end{equation}
This format encompasses a broad class of feasibility problems. Our main duality results focus on the specific case where both \( R \) and \( S \) are polyhedral cones, such as the non-negative orthant, a linear subspace, or a linear transformation of a Cartesian product of these cones.

Since our interest lies in the behavior of the {\em solution set} for feasibility problems, we will utilize standard {\em set-valued mappings} terminology and notation, as discussed in the popular textbooks by Rockafellar and Wets~\cite{RockWets09}, Kahn et al.~\cite{KhanTZ16}, Ioffe~\cite{Ioffe17}, and Mordukhovich~\cite{Mord18}.  We briefly recall the terminology and notation of set-valued mappings that suffice for our purposes in Section~\ref{sec.Hoffman.duality}.
Set-valued mappings frequently arise in optimization, nonsmooth analysis, variational inequalities, and equilibrium models. In contrast to point-to-point mappings, set-valued mappings map points to sets. This feature allows for more complex and flexible relationships.  Set-valued mappings are a natural tool for studying generalized constraints, subdifferentials, and optimality conditions where it is critical to get a grasp on the entire solution set to a problem.
Set-valued mappings are particularly prevalent in the literature on Hoffman constants and error bounds, as highlighted in ~\cite{CamaCGP23,CamaCGP22,PenaVZ21}.

Define the {\em solution mapping} $P_{A,R,S}:\R^m\rightrightarrows \R^n$  corresponding to the feasibility problem~\eqref{eq.onesided.systems.general} as follows
\begin{equation}\label{eq.PA.gral}
P_{A,R,S}(b):=\{x\in R: Ax-b\in S\}.
\end{equation}
In other words, $P_{A,R,S}$ maps each $b\in \R^m$ to the (possibly empty) set of solutions $P_{A,R,S}(b)\subseteq\R^n$ to the feasibility problem~\eqref{eq.onesided.systems.general}.  

The role of norms and their duals is central to our developments.  To that end, 
suppose $\R^n$ and $\R^m$ are endowed with some (primal) norms $\|\cdot\|$.  Let $\dist(\cdot,\cdot)$ denote the point-to-set distance defined by the norm $\|\cdot\|$.  
Let $\|\cdot\|^*$ denote the dual norms in $\R^n$ and $\R^m$ and let $\dist^*(\cdot,\cdot)$ denote the point-to-set distance defined by the norm $\|\cdot\|^*$. Also, given a matrix $A \in \R^{m \times n}$ and sets $R \subseteq \R^n, \, S\subseteq \R^m$, we will write $A(R)-S$ to denote the set $\{Ax-s: x \in R, \; s\in S\}$. 

Hoffman's Lemma~\cite{Hoff52} implies that there exists a finite constant $H$ that depends only on $(A,R,S)$ such that for all 
$x\in R$ and $b\in A(R)-S$ 
\begin{equation}\label{eq.hoff.PA}
\dist(x,P_{A,R,S}(b)) \le H \cdot \dist(Ax-b,S).
\end{equation}

Throughout the paper we will write $\H(P_{A,R,S})$ to denote the sharpest constant $H$ satisfying~\eqref{eq.hoff.PA}, that is,
\begin{equation}\label{eq.def.hoffman.PA}
\H(P_{A,R,S}) = \sup_{b\in A(R)-S \atop x\in R\setminus P_{A,R,S}(b)} \frac{\dist(x,P_{A,R,S}(b))}{\dist(Ax-b,S)}.
\end{equation}

In the special case when $R$ and $S$ are convex cones, there is a natural dual counterpart of~\eqref{eq.onesided.systems.general}, namely, the feasibility problem
\begin{equation}\label{eq.onesided.systems.general.dual}
 \begin{array}{r}
c-A\transp y \in R^*\\
y \in S^*,
\end{array}
\end{equation}
where $R^*\subseteq\R^n$ and $S^*\subseteq \R^m$ are the dual cones of $R$ and $S$ respectively.

Observe that the problem~\eqref{eq.onesided.systems.general.dual} has the same format as~\eqref{eq.onesided.systems.general} but is defined by the triple $(A\transp, S^*, -R^*)$ in lieu of $(A,R,S)$.  In particular, the solution mapping $P_{A\transp,S^*,-R^*}: \R^n\rightrightarrows \R^m$ to~\eqref{eq.onesided.systems.general.dual} is
\begin{equation}\label{eq.QA.gral}
P_{A\transp,S^*,-R^*}(c)=\{y\in S^*: A\transp y - c\in -R^*\} = \{y\in S^*: c-A\transp y\in R^*\}.
\end{equation}
Let $\ri(\cdot)$ denote the relative interior of a set.
We show the following duality result (Corollary~\ref{cor.onesided.dual}):  If the {\em primal} Slater condition $Ax \in \ri(S),\, x \in \ri(R)$ holds for some $x\in \R^n$ then $\H(P_{A,R,S}) \le \H(P_{A\transp,S^*,-R^*})$.  In natural symmetric fashion,  
if the {\em dual} Slater condition $-A\transp y \in \ri(R^*),\, y \in \ri(S^*)$ holds for some $y\in \R^m$ then $ \H(P_{A\transp,S^*,-R^*}) \le \H(P_{A,R,S})$. 
The building blocks for Corollary~\ref{cor.onesided.dual} are a generic {\em Hoffman duality inequality} (Theorem~\ref{thm.dual.ineq}) 
and a  characterization of Hoffman constants (Theorem~\ref{thm.Hoffman.char}).

\medskip

The above duality result (Corollary~\ref{cor.onesided.dual}) applies to the case when $R$ and $S$ are cones but its building blocks, namely Theorem~\ref{thm.Hoffman.char} and Theorem~\ref{thm.dual.ineq}, also yield an interesting duality result in the following non-conic case. Suppose $R=\B\subseteq \R^n$ is a box and $S = L\subseteq \R^m$ is a linear subspace with orthogonal complement $L^\perp\subseteq \R^m$. In Theorem~\ref{thm.box.dual} we show the {\em  identity} 
$\H(P_{A,\B,L}) = \H(P_{A\transp,L^\perp,-\B})$ between the Hoffman constant of the solution mapping $P_{A,\B,L}:\R^m\rightrightarrows \R^n$ to the feasibility problem
\begin{equation}\label{eq.box-constrained}
\begin{array}{r}
Ax-b \in L\\
x \in \B
\end{array}
\end{equation}
and the Hoffman constant of the solution mapping $P_{A\transp,L^\perp,-\B}:\R^n\rightrightarrows \R^m$ to the feasibility problem
\begin{equation}\label{eq.box-constrained.dual}
\begin{array}{r}
c-A\transp y \in \B\;\;\;\\
y \in L^\perp.
\end{array}
\end{equation}
It is important to notice that~\eqref{eq.box-constrained} and~\eqref{eq.box-constrained.dual} are not dual to each other. 

\medskip

In addition to the above duality properties, we also show a striking identity (Theorem~\ref{thm.chi.hoffman}) between the Hoffman constants for box-constrained feasibility problems~\eqref{eq.box-constrained} and the chi condition measures for weighted least squares problems~\cite{BobrV01,Fors96,ForsS01,Olea90,Stew89} when the underlying norms are Euclidean.  The identity of these quantities is particularly tantalizing because the Hoffman constant and the chi measures are constructed and intended to capture conditioning properties of fundamentally different problems.  Hoffman constants are concerned with error bound properties for linear inequalities and are oblivious to scaled projections.  By contrast, the chi measures are concerned with norms of scaled projections and are oblivious to inequalities.  Our identity is surprising and rectifies a long-standing misconception in the literature on the relation between chi measures and Hoffman bounds~\cite{HoT002}.  More precisely, some of our results (Theorem~\ref{thm.chi.hoffman} and Lemma~\ref{cor.signed}) readily show that, contrary to the explicit statement at the end of~\cite[Section 4]{HoT002}, the inequality between the two quantities in~\cite[Theorem 4.6]{HoT002} can be strengthened to equality.

\section{Hoffman duality inequality}
\label{sec.Hoffman.duality}
This section outlines the foundational elements of our developments, namely a key characterization of Hoffman constants (Theorem~\ref{thm.Hoffman.char}) and the {\em Hoffman duality inequality} (Theorem~\ref{thm.dual.ineq}).  The latter in turn yields Corollary~\ref{cor.onesided.dual} which establishes our main duality result concerning the feasibility problems~\eqref{eq.onesided.systems.general} and~\eqref{eq.onesided.systems.general.dual}

To introduce and motivate our notation, we describe a general format for formulating feasibility problems.  Suppose the triple $(\R^p,\R^q,G)$ is such that $G \subseteq \R^p \times \R^q$ and consider the following generic feasibility problem: 
\[
\text{ given } u \in \R^p \text{ find } v \in \R^q \text{ such that } (u,v) \in G. 
\]
The natural {\em solution mapping} $\Phi_G: \R^p\rightrightarrows \R^q$ for this feasibility problem is as follows
\[
\Phi_G(u):=\{v\in \R^q: \text{ such that } \; (u,v) \in G\}.
\]
 The following example illustrates how the solution mapping $P_{A,R,S}$ defined via~\eqref{eq.PA.gral} in Section~\ref{sec.overview} is an instance of the solution mapping $\Phi_G$ for a suitable chosen $G$. 
 
\medskip
 
\begin{example} Suppose $R\subseteq \R^n$ and $S\subseteq\R^m$ are reference polyhedra and  $A\in \R^{m\times n}$. For
$G = \{(Ax-s,x): x\in R, s\in S\}\subseteq \R^m\times \R^n$ the solution mapping $\Phi_G$ is precisely  $P_{A,R,S}$, that is, the  solution mapping defined in~\eqref{eq.PA.gral}.
\end{example}

\medskip

The above construction highlights the natural connection between feasibility problems and set-valued mappings and introduces a key concept: The set $G$ corresponds to the {\em graph} of the set-valued mapping $\Phi_G$.  In the sequel we will rely on the basic concepts of graph, domain, and image of set-valued mappings as described in~\cite[Chapter 5]{RockWets09} and~\cite[Chapter 1]{KhanTZ16}. 

Suppose $\Phi:\R^p\rightrightarrows \R^q$ is a set-valued mapping, that is, $\Phi(u) \subseteq \R^q$ for each $u\in \R^p$. The {\em graph} of $\Phi$ is 
\begin{equation}
\label{eq:graphdef}
\graph(\Phi) = \{(u,v)\in \R^p\times \R^q: v\in \Phi(u)\}.
\end{equation}
Observe that $\Phi$ is precisely the solution mapping of the following feasibility problem,
\[
\text{ given } u \in \R^p \text{ find } v \in \R^q \text{ such that } (u,v) \in \graph(\Phi). 
\]
The domain $\dom(\Phi)\subseteq\R^p$, and image $\Image(\Phi)\subseteq\R^q$, of $\Phi$ are as follows 
\[
\dom(\Phi) := \{u\in \R^p: \Phi(u) \ne \emptyset\}, \; \Image(\Phi):=\bigcup_{u\in \R^p} \Phi(u).
\]
It is easy to see that the $\dom(\Phi)$ and $\Image(\Phi)$ are the projections of $\graph(\Phi)$ onto $\R^p$ and $\R^q$ respectively.

The inverse mapping $\Phi^{-1}:\R^q\rightrightarrows \R^p$ is defined via $u\in \Phi^{-1}(v) \Leftrightarrow v\in \Phi(u)$.  Observe that
\[
\graph(\Phi^{-1}) = \{(v,u)\in \R^q\times \R^p: (u,v) \in \graph(\Phi)\},
\]
and consequently $\dom(\Phi) = \Image(\Phi^{-1})$ and $
\Image(\Phi) = \dom(\Phi^{-1})$.

A set-valued mapping $\Phi:\R^p\rightrightarrows \R^q$ is {\em polyhedral} if  $\graph(\Phi)\subseteq \R^p\times \R^q$ is a polyhedron.  Suppose $\Phi:\R^p\rightrightarrows \R^q$
is a polyhedral mapping and each of the spaces $\R^p$ and $\R^q$ is endowed with some norm $\|\cdot\|$. 
Define the {\em Hoffman constant}  $\H(\Phi)$ as the sharpest $H$ such that for all $u\in \dom(\Phi)=\Image(\Phi^{-1})$ and all $v\in \Image(\Phi)= \dom(\Phi^{-1})$ the following inequality holds
\[
\dist(v,\Phi(u)) \le H \cdot \dist(u,\Phi^{-1}(v)).
\]
In other words, the Hoffman constant $\H(\Phi)$ is 
\begin{equation}\label{eq.def.Hoffman}
\H(\Phi) := \sup_{u\in\dom(\Phi)\atop 
v\in \Image(\Phi)\setminus\Phi(u)}\frac{\dist(v,\Phi(u))}{\dist(u,\Phi^{-1}(v))}.
\end{equation}
Note that the right-hand side expression in~\eqref{eq.def.Hoffman} is well-defined.  Indeed, since $\Phi$ is polyhedral, it follows that $\Phi^{-1}(v)$ is nonempty and closed for all $v\in \Image(\Phi)$.  Thus $\dist(u,\Phi^{-1}(v)) > 0$ whenever $u\in\dom(\Phi)$ and $v\in \Image(\Phi)\setminus \Phi(u)$.

\medskip

A straightforward verification shows that for $\Phi = P_{A,R,S}$ the definition of the Hoffman constant $\H(P_{A,R,S})$ in~\eqref{eq.def.Hoffman} indeed matches the one in~\eqref{eq.def.hoffman.PA}.  
\medskip

We should clarify a minor but important detail concerning the notation $\H(\Phi)$ of the Hoffman constant~\eqref{eq.def.Hoffman}.  In our previous article~\cite{PenaVZ21}, this constant was denoted $\H(\Phi^{-1})$.  Thus our references to~\cite{PenaVZ21}, such as Theorem~\ref{thm.Hoffman.char} below, have been adjusted accordingly.  Although this is only a choice of notation, and thus of no fundamental relevance, the notation in~\eqref{eq.def.Hoffman} is more convenient as it emphasizes the role of the solution mapping $\Phi$.  It is also consistent with the notation used in other recent articles~\cite{CamaCGP23,CamaCGP22}.

\medskip

We next recall some key concepts that lie at the heart of our duality developments, namely {\em sublinear mappings} and the {\em upper adjoint}.  
A set-valued mapping $\Phi:\R^p\rightrightarrows \R^q$ is {\em sublinear} if $0 \in \Phi(0)$ and $\lambda \Phi(x) + \lambda' \Phi(x') \subseteq \Phi(\lambda x + \lambda' x')$ for all $\lambda, \lambda' \ge 0$ and $x, x'\in \R^p$. In other words,
a set-valued mapping is sublinear if and only if $\graph(\Phi)$ is a convex cone. Set-valued sublinear mappings are also known as convex processes~\cite{Borw86,Robi72}.

Suppose $\Phi:\R^p\rightrightarrows \R^q$ is sublinear.  The {\em upper adjoint} $\Phi^*:\R^q\rightrightarrows
\R^p$ is defined as 
\begin{equation}\label{eq.upper.adjoint}
z\in \Phi^*(w) \Leftrightarrow  \ip{z}{u}\le\ip{w}{v} \text{ for all } (u,v) \in \graph(\Phi).
\end{equation}
Observe that 
\begin{equation}\label{eq.upperAdj.graph}
\graph(\Phi^*) = \{(w,z): (-z,w) \in \graph(\Phi)^*\}.
\end{equation}

The following example illustrates the upper adjoint in a particularly relevant case for our developments.

\begin{example}\label{prop.adjoint} 
Suppose $R\subseteq \R^n, S\subseteq\R^m$ are reference polyhedral cones and  $A\in \R^{m\times n}$.  Then the mappings $P_{A,R,S}$ and $P_{A\transp,S^*,-R^*}$ defined via~\eqref{eq.PA.gral} and~\eqref{eq.QA.gral} are sublinear and
\[(P_{A,R,S})^* = P_{A\transp,S^*,-R^*}.\] 
\end{example}

To ease notation, in the special case when  $A\in \R^{m\times n}, R= \R^n_+,$ and $S=\{0\}\in \R^m$ we shall write $P_A$ as shorthand for $P_{A,\R^n_+,\{0\}}$ and $P_{A}^*$ as shorthand for its upper adjoint $(P_{A,\R^n_+,\{0\}})^* = P_{A\transp, \R^m,-\R^n_+}$.  In other words, $P_A:\R^m\rightrightarrows \R^n$ is the mapping defined via 
\begin{equation}\label{eq.PA}
P_A(b) = \{x\ge 0: Ax=b\}
\end{equation}
and $P_A^*:\R^n\rightrightarrows \R^m$ is the mapping
defined via 
\begin{equation}\label{eq.QA}
P_A^*(c) = \{y\in \R^m: A\transp y \le c\}.
\end{equation}

Our main developments hinge on the characterization of  Hoffman constants in Theorem~\ref{thm.Hoffman.char} below.  This characterization is crucial for this article and has broader implications that extend beyond our specific context. In fact, the code available in the following repository utilizes Theorem~\ref{thm.Hoffman.char} to compute Hoffman constants:
\begin{center}
\href{https://github.com/javi-pena/HoffmanCode}{https://github.com/javi-pena/HoffmanCode}.
\end{center}
This code has been used in~\cite{xia2020globally} to obtain integer programming formulations of quadratic programs.

Theorem~\ref{thm.Hoffman.char} relies on specific collections of tangent cones associated with a polyhedral set-valued mapping. We will also use these collections of tangent cones in some of our developments.  Suppose $\Phi:\R^p\rightrightarrows \R^q$
is a polyhedral mapping.  For $(u,v)\in \graph(\Phi)$, or equivalently for $v\in \Phi(u)$,  let $T_{(u,v)}(\Phi)$ denote the tangent cone to $\graph(\Phi)$ at $(u,v)$, that is,
\[
T_{(u,v)}(\Phi) = \{(w,z)\in \R^p\times \R^q: (u,v) + t(w,z) \in \graph(\Phi) \text{ for some } t>0\}.
\]  
Let $\T(\Phi)$ denote the collection of all tangent cones to $\graph(\Phi)$, that is,
\[
\T(\Phi):=\{T_{(u,v)}(\Phi): (u,v) \in \graph(\Phi)\}.
\]
For $T\in \T(\Phi)$ let $\Phi_T:\R^p\rightrightarrows \R^q$ be the sublinear mapping defined via
\[
z\in \Phi_T(w) \Leftrightarrow (w,z) \in T.
\]
In other words, $\Phi_T:\R^p\rightrightarrows \R^q$ is the sublinear mapping that satisfies $\graph(\Phi_T) = T$. 
Let $\S(\Phi)\subseteq \T(\Phi)$ be the following smaller collection of tangent cones
\[
\S(\Phi):=\{T\in
\T(\Phi): \dom(\Phi_T) \text{ is a linear subspace}\}.
\]
The following example illustrates the sets $\T(\Phi)$ and 
$\S(\Phi)$ for the solution mappings $P_A$ and  $P_A^*$.

\begin{example}
\label{ex.tangent}
Suppose $A\in \R^{m\times n}$, and consider the solution mappings $P_A$ and  $P_A^*$ defined in~\eqref{eq.PA} and~\eqref{eq.QA}.
\begin{enumerate}\itemsep2pt
\item[(a)] Suppose $\Phi = P_{A}$. Then each member of  $\T(\Phi)$ is of the form $T_I := \{(Ax,x): x_I \ge 0\}$ for some $I\subseteq [n]$.  Furthermore, $T_I\in \S(\Phi)$ if and only if the following  Slater condition holds: $Ax = 0, x_I > 0$ for some $x\in \R^n$.

\item[(b)] Suppose $\Phi = P_A^*$. Then each member of $\T(\Phi)$ is of the form $T^I := \{(A\transp y +s,y): y\in\R^m, \; s_I\ge 0\}$ for some $I\subseteq [n]$.  Furthermore,  $T^I\in \S(\Phi)$ if and only if the following Slater condition holds: $(A\transp y)_I < 0$ for some $y\in \R^m$.
\end{enumerate}
\end{example}

 Define the {\em norm} of a sublinear mapping $\Phi:\R^p\rightrightarrows \R^q$ as follows
\begin{equation}\label{eq.norm.mapping}
\|\Phi\|=\max_{u\in\dom(\Phi) \atop \|u\|\le 1} \min_{v\in \Phi(u)}\|v\|.
\end{equation}
The upper adjoint $\Phi^*$ is a mapping between the dual spaces of $\R^q$ and $\R^p$.  According to~\eqref{eq.norm.mapping} the norm of $\Phi^*$ is as follows
\begin{equation}\label{eq.dual.norm}
\|\Phi^*\|^* = \max_{w\in \dom(\Phi^*)\atop \|w\|^* \le 1} \min_{z\in\Phi^*(w)}\|z\|^*.
\end{equation}

Observe that when $\Phi$ is a linear mapping,  the norm $\|\Phi\|$ defined via~\eqref{eq.norm.mapping} is precisely the usual operator norm of $\Phi$. For general sublinear mappings, we should note that the norm defined via~\eqref{eq.norm.mapping} is due to Robinson~\cite{Robi72} and is slightly different from the popular {\em inner} norm discussed and used in~\cite{Mord18,RockWets09}, namely,
\[
\|\Phi\|^-=\max_{u\in\R^p\atop\|u\|\le 1} \min_{v\in \Phi(u)}\|v\|.
\]
It is evident that $\|\Phi\| = \|\Phi\|^-$
 when $\dom(\Phi) = \R^p$.  On the other hand, when $\Phi:\R^p\rightrightarrows \R^q$ is a polyhedral sublinear mapping and $\dom(\Phi) \ne \R^p$, it holds that
 $\|\Phi\|^-= +\infty$ whereas  $\|\Phi\|< +\infty$.  The latter finiteness property is crucial for the characterization of Hoffman constants stated in Theorem~\ref{thm.Hoffman.char} below.  Theorem~\ref{thm.Hoffman.char} is a key foundational block for our developments.  It is an extension of previous results in~\cite{PenaVZ21}  tailored to the main purposes of this paper.

\medskip
\begin{theorem}\label{thm.Hoffman.char}
Let $\Phi:\R^p\rightrightarrows \R^q$ be a polyhedral mapping.
\begin{enumerate}
\item[(a)]  Suppose $\mathcal S \subseteq \T(\Phi)$ satisfies the following {\em bounding property:} For all $T\in \S(\Phi)$ there exists $T'\in \mathcal S$ such that $\|\Phi_T\|\le \|\Phi_{T'}\|.$ Then
\begin{equation}\label{cor.eq.hoffman.char}
 \H(\Phi)=\max_{T\in \mathcal S}
\|\Phi_T\|.
\end{equation}
In particular, if $\dom(\Phi)$ is a linear subspace then $\mathcal S := \{\graph(\Phi)\}$ satisfies the bounding property and thus
\begin{equation}\label{eq.hoffman.special}
 \H(\Phi)=\|\Phi\|.
\end{equation}
\item[(b)] For all $T\in \T(\Phi)$ it holds that $\H(\Phi_T) \le \H(\Phi)$.
\end{enumerate}
\end{theorem}
\begin{proof}
$\;$
\begin{enumerate}
\item[(a)] From~\cite[Theorem 1 and Lemma 1]{PenaVZ21} it follows that 
\begin{equation}\label{eq.hoffman.char}
 \H(\Phi)=\max_{T\in \S(\Phi)}
\|\Phi_T\|=\max_{T\in \mathcal \T(\Phi)} \|\Phi_T\|.
\end{equation}
The bounding property and~\eqref{eq.hoffman.char} imply that
\[
\H(\Phi) = \max_{T\in \S(\Phi)} \|\Phi_T\| \le \max_{T\in \mathcal S}
\|\Phi_T\| \le \max_{T\in \mathcal \T(\Phi)}\|\Phi_T\| = \H(\Phi).
\]
The first and last steps above follow from~\eqref{eq.hoffman.char}.  The first inequality follows from the bounding property of $\mathcal S$ and the second one follows from the assumption that $\mathcal S \subseteq \mathcal \T(\Phi)$.
Thus~\eqref{cor.eq.hoffman.char} follows.

When $\dom(\Phi)$ is a linear subspace, it readily follows that $\dom(\Phi) = \dom(\Phi_T)$ for all $T\in \T(\Phi)$ and hence $\mathcal S := \{\graph(\Phi)\}$ satisfies the bounding property.  The latter fact and~\eqref{cor.eq.hoffman.char} imply~\eqref{eq.hoffman.special}.
\item[(b)] Since $\graph(\Phi)$ is a polyhedron, it follows that $\T(\Phi_T) \subseteq \T(\Phi).$ Therefore
\[
\H(\Phi_T) = \max_{T'\in \T(\Phi_T)}\|\Phi_{T'}\| \le  \max_{T'\in \T(\Phi)}\|\Phi_{T'}\| = \H(\Phi).
\]
The first and last steps above follow from  the second identity in~\eqref{eq.hoffman.char} applied to $\Phi_T$ and $\Phi$. 
The inequality in the middle step holds because $\T(\Phi_T) \subseteq \T(\Phi).$
\end{enumerate}

\end{proof}

The following example illustrates Theorem~\ref{thm.Hoffman.char} for the solution mappings $P_A$ and $P_A^*$. This example relies on the description of the tangent cones detailed in Example~\ref{ex.tangent}.

\begin{example}\label{ex.char} Suppose $A\in \R^{m\times n}$, and consider the solution mappings $P_A$ and  $P_A^*$ defined in~\eqref{eq.PA} and~\eqref{eq.QA}.
\begin{enumerate}\itemsep2pt
\item[(a)] If the primal Slater condition $Ax=0, x > 0$ holds for some $x\in \R^n$ then $\dom(P_{A}) = A(\R^n) \subseteq \mathbb{R}^m$ is a linear subspace and thus Theorem~\ref{thm.Hoffman.char} implies that
\begin{equation}\label{eq.hoff.A}
\H(P_{A}) = \|P_{A}\| = \max_{z\in A(\R^n)\atop \|z\| \le 1} \min_{x\ge 0 \atop Ax = z} \|x\|. 
\end{equation}
Let  $B = \{z \in A(\R^n): \|z\| \le 1\}$ and $S := \{A x : x \ge 0, \|x\| \le 1\} \subseteq A(\R^n)$. Observe that
\[
\max_{z\in A(\R^n)\atop \|z\| \le 1} \min_{x\ge 0 \atop Ax = z} \|x\| = \min\{t: B\subseteq t S\} = \frac 1{\inf \{t: t B \not \subseteq S\}}.
\]
Hence~\eqref{eq.hoff.A} yields
\begin{equation}\label{eq.PA.char}
\begin{aligned}
\H(P_{A}) &=  \frac 1{\inf \{t: t B \not \subseteq S\}} \\&= \frac 1{\inf \{\|A y\|: Ay \notin \{Ax: x \ge 0, \|x\| \le 1\}\}}.
\end{aligned}
\end{equation}
The denominator in the last expression in~\eqref{eq.PA.char} is the radius of the largest ball in the space $A(\R^n)$ centered at the origin and contained in the set $\{Ax: x\ge 0, \|x\|\le 1\}$. The primal Slater condition ensures that this radius is strictly positive and thus $\H(P_A)$ is finite, as it should be.

More generally, let $\mathcal I$ denote the collection of maximal sets $I\subseteq [n]$ (with respect to inclusion) such that the Slater condition $Ax=0, x_I > 0$ holds for some $x\in \R^n$.  Then the collection of tangent cones $\{T_I: I\in \mathcal I\} \subseteq \graph(P_{A})$ where $T_I = \{(Ax,x): x_I \ge 0\}$ satisfies the bounding property and consequently
\[
\H(P_{A}) = \max_{I\in \mathcal I} \frac{1}{\inf\{\|Ay\|: Ay \not \in  \{Ax: x_I\ge 0, \|x\|\le 1\}\}}.
\]

\item[(b)] If the dual Slater condition $A\transp y < 0$ holds for some $y\in \R^m$ then $\dom(P_A^*) = \R^n$ is a linear subspace and thus Theorem~\ref{thm.Hoffman.char} implies that
\begin{equation}\label{eq.char.QA}
\begin{aligned}
\H(P_A^*) &= \|P_A^*\|^* 
\\&= \max_{z\in\R^n\atop \|z\|^* \le 1} \min_{A\transp y \le z} \|y\|^* 
\\&= \max_{z\in\R^n\atop \|z\|^* \le 1} \max_{x \ge 0 \atop \|Ax\| \le 1} -\ip{z}{x} 
\\&=  \max_{x \ge 0 \atop \|Ax\| \le 1} \max_{z\in\R^n\atop \|z\|^* \le 1}-\ip{z}{x} 
\\&= \max\{\|x\|: x \ge 0, \|Ax\| \le 1\}.
\end{aligned}
\end{equation}
The third step in~\eqref{eq.char.QA} follows by Fenchel duality and the last step follows from the following duality property of norms
\[
\|x\| = \max_{\|y\|^*\le 1} \ip{y}{x}.
\]  
The last expression in~\eqref{eq.char.QA} has the following evident geometric interpretation: It is the largest element in the set $\{x \in \R^n: x \ge 0, \|Ax\| \le 1\}$.  Observe that this quantity is finite and thus so is $\H(P_A^*)$, as it should be, because the dual Slater condition implies that $Ax\ne 0$ for any $x\ge 0, \, x\ne 0$.

More generally, let $\mathcal I$ be the collection of maximal sets $I\subseteq [n]$ (with respect to inclusion) such that the Slater condition $(A\transp y)_I < 0$ holds for some $y\in \R^m$.  Then the collection of tangent cones $\{T^I: I\in \mathcal I\} \subseteq \graph(P_A^*)$ where $T^I = \{(A\transp y+s,y): s_I \ge 0\}$ satisfies the bounding property and consequently
\[
\H(P_A^*) =
 \max_{I\in \mathcal I} \; \max\{\|x\|: x \in \R^I, \, x \ge 0, \|A_Ix\| \le 1\}.
\]

\end{enumerate}
\end{example}

\medskip

In the special case where $\Phi:\R^p\rightarrow \R^q$ is a linear mapping, it is well known that $\|\Phi\| = \|\Phi^*\|^*$.  Furthermore, in this special case it also holds that $\H(\Phi) =\H(\Phi^*)$
since Theorem~\ref{thm.Hoffman.char} implies that
$\H(\Phi) = \|\Phi\|$ and $\H(\Phi^*) = \|\Phi^*\|^*$. 
Although neither $\|\Phi\| = \|\Phi^*\|^*$ nor $\H(\Phi) =\H(\Phi^*)$ necessarily holds when $\Phi:\R^p\rightrightarrows \R^q$ is a general polyhedral sublinear mapping,  the following theorem shows that a suitable combination of these two identities, namely $\|\Phi\|\le \H(\Phi^*)$, holds.
As we show later, Theorem~\ref{thm.dual.ineq} automatically unveils several interesting duality properties of Hoffman constants.

\begin{theorem}[Hoffman duality inequality]\label{thm.dual.ineq} Suppose $\Phi:\R^p\rightrightarrows \R^q$ is a polyhedral sublinear mapping.  Then
\[
\|\Phi\| \le\H(\Phi^*).
\]
In particular,  if $\dom(\Phi)$ is a linear subspace then
\[
\H(\Phi) \le \H(\Phi^*).
\]
\end{theorem}
\begin{proof} 
By Theorem~\ref{thm.Hoffman.char}, it suffices to show that there exists $T\in \T(\Phi^*)$ such that $\|\Phi^*_T\|^* \ge \|\Phi\|$.  To that end, let $\bar u\in \dom(\Phi)$ with $\|\bar u\|=1$ be such that 
\[
\|\Phi\| = \min_{v\in \Phi(\bar u)}\|v\|.
\]
The existence of such $\bar u$ follows from~\eqref{eq.norm.mapping}.
We will show that there exists $T\in \T(\Phi^*)$ and $\hat w \in \dom(\Phi^*_T)$ with $\|\hat w\|^*\le 1$ such that
\begin{equation}\label{eq.to.finish}
z\in (\Phi^*_T)(\hat w) \Rightarrow \|z\|^* \ge \min_{v\in \Phi(\bar u)}\|v\|.
\end{equation}
This will finish the proof because~\eqref{eq.norm.mapping} and~\eqref{eq.to.finish} imply that 
\[
\|\Phi^*_T\|^* = \max_{w\in \dom(\Phi_T^*)\atop \|w\|^*\le 1}\min_{z\in (\Phi^*_T)(w)} \|z\|^* \ge \min_{z\in (\Phi^*_T)(\hat w)} \|z\|^* \ge \min_{v\in \Phi(\bar u)}\|v\| = \|\Phi\|.
\]
We next construct $T\in\T(\Phi^*)$ and $\hat w\in \dom(\Phi^*_T)$ with $\|\hat w\|^*\le1$ such that~\eqref{eq.to.finish} holds.  
To that end, first observe that 
\begin{equation}\label{eq.step.1}
\min_{v\in \Phi(\bar u)}\|v\| = \min_{v\in \Phi(\bar u)} \max_{\|w\|^* \le 1} \ip{w}{v}.
\end{equation}  
Second, apply Sion's minimax theorem~\cite{Sion58} to get
\begin{equation}\label{eq.step.2}
\min_{v\in \Phi(\bar u)} \max_{\|w\|^* \le 1} \ip{w}{v} = \max_{\|w\|^* \le 1} \min_{v\in \Phi(\bar u)} \ip{w}{v}.
\end{equation}
Third, since $\Phi$ is a polyhedral sublinear mapping, 
$K:=\graph(\Phi)$ is a polyhedral convex cone and thus convex duality implies that for each $w\in \R^n$
\begin{equation}\label{eq.step.3}
\min_{v\in \Phi(\bar u)} \ip{w}{v} = \min_{v \in \R^q\atop (\bar u,v)\in K} \ip{w}{v} = \max_{z\in \R^p\atop (-z,w)\in K^*} \ip{z}{\bar u} 
 = \max_{z\in \Phi^*(w)} \ip{z}{\bar u}.
\end{equation}
The first step in~\eqref{eq.step.3} follows from~\eqref{eq:graphdef} and the third one follows from~\eqref{eq.upperAdj.graph}. For the second step, notice that 
\[
\min_{v \in \R^q\atop (\bar u,v)\in K} \ip{w}{v} = 
 \min \left \{ \left \langle \matr{ 0 \\w}, 
\matr{u \\v}
\right \rangle:
\matr{ I & 0} \matr{ u \\v} = \bar u, \matr {u \\v } \in K \right \},
\]
whose conic dual, after associating the dual variables $z \in \R^p$ with the constraints $Iu + 0v = \bar u$, is
\[
\max \left \{ \ip{z}{\bar u}: \matr {0 \\w } - \matr{ I \\ 0 } z \in K^* \right \} = \max_{z\in \R^p\atop (-z,w)\in K^*} \ip{z}{\bar u}.
\]
Since $K$ is polyhedral, the second step in~\eqref{eq.step.3} follows from strong duality.

Putting together~\eqref{eq.step.1},~\eqref{eq.step.2}, and~\eqref{eq.step.3} we get
\begin{equation}\label{eq.pair}
\min_{v\in \Phi(\bar u)}\|v\| =\min_{v\in \Phi(\bar u)} \max_{\|w\|^* \le 1} \ip{w}{v}
 =  \max_{\|w\|^* \le 1\atop z\in \Phi^*(w)} \ip{z}{\bar u}.
\end{equation}
Let $\bar v$ and $(\bar z,\bar w)$ be optimal solutions to the left-most and right-most problems in~\eqref{eq.pair}. Then $\bar v \in \Phi(\bar u),\; \bar z\in \Phi^*(\bar w),$ and 
\begin{equation}\label{eq.orthogonal}
\|\bar v\| = \ip{\bar w}{\bar v} = \ip{\bar z}{\bar u}.
\end{equation}
Let $T:=T_{(\bar w,\bar z)}(\Phi^*) \in \T(\Phi^*)$.  Observe that 
$(-\bar w,-\bar z)\in T$.  In addition, for all $(-\bar w,z) \in T$ we have $\bar z + t z \in \Phi^*(\bar w-t\bar w)$ for $t > 0$ sufficiently small and hence the construction of the upper adjoint $\Phi^*$ and~\eqref{eq.orthogonal} yield
\[
\ip{\bar z + tz}{\bar u} \le \ip{\bar w - t\bar w}{\bar v} \Rightarrow \ip{z}{\bar u} \le -\ip{\bar w}{\bar v}= -\|\bar v\|. 
\]
In other words, if $z\in \Phi_T^*(-\bar w)$ then $\ip{z}{\bar u} \le -\|\bar v\|$.
Thus $\hat w:=-\bar w \in \dom(\Phi_T^*)$ satisfies $\|\hat w\|^* = \|\bar w\|^*\le 1$, and for all $z\in \Phi_T^*(\hat w)$ we have
\[
\|z\|^* \ge -\ip{z}{\bar u} \ge \|\bar v\| = \min_{v\in \Phi(\bar u)}\|v\|. 
\]
Therefore~\eqref{eq.to.finish} holds.
\end{proof}
Applying Theorem~\ref{thm.dual.ineq} to the mappings $P_{A,R,S}$ and $P_{A\transp,S^*,-R^*}$ we obtain the following relationship between the Hoffman constants of the primal-dual pair of feasibility problems~\eqref{eq.onesided.systems.general} and~\eqref{eq.onesided.systems.general.dual}.  
\begin{corollary}\label{cor.onesided.dual}
 Suppose $R\subseteq \R^n, S\subseteq \R^m$ are polyhedral cones and $A\in \R^{m\times n}$.  Consider the solution mappings $P_{A,R,S}:\R^m \rightrightarrows \R^n$ and  $P_{A\transp,S^*,-R^*}:\R^n\rightrightarrows \R^m$  defined via~\eqref{eq.PA.gral} and~\eqref{eq.QA.gral}.
\begin{enumerate}[label = (\alph*)]
\item If the primal Slater condition $Ax \in \ri(S), x\in \ri(R)$ holds for some $x\in \R^n$ then $\H(P_{A,R,S}) \le \H(P_{A\transp,S^*,-R^*})$.
\label{it:itema}
\item If the dual Slater condition $-A\transp y \in \ri(R^*), \; y\in \ri(S^*)$ holds for some $y\in \R^m$ then $\H(P_{A\transp,S^*,-R^*})\le \H(P_{A,R,S})$.
\label{it:itemb}
\end{enumerate}
\end{corollary}

Corollary~\ref{cor.onesided.dual} is an immediate consequence of Theorem~\ref{thm.dual.ineq} and the following lemma.

\begin{lemma}\label{lemma.slater} 
Suppose $R\subseteq \R^n, S\subseteq \R^m$ are polyhedral cones and $A\in \R^{m\times n}$. Let $P_{A,R,S}:\R^m \rightrightarrows \R^n$ be the solution mapping defined via~\eqref{eq.PA.gral}. Then $\dom(P_{A,R,S})\subseteq \R^m$ is a linear subspace if and only if the primal Slater condition $Ax \in \ri(S), x\in \ri(R)$ holds for some $x\in \R^n$. 
\end{lemma}
\begin{proof}
Since $\dom(P_{A,R,S}) = A(R) - S$, we need to show that 
\[
A(R) - S \text{ is a linear subspace } \Leftrightarrow \;  A x \in \ri(S) \text{ for some }  x\in \ri(R).
\]
First we show ``$\Rightarrow$''. To that end, pick $x_0\in \ri(R), s_0\in \ri(S)$ and let $z_0 = Ax_0 - s_0 \in A(R)-S$.  Since the latter set is a linear subspace, we also have $-z_0 \in A(R)-S$, that is, $-z_0 = A\hat x - \hat s$ for some $\hat x\in R$ and $\hat s \in S$.  Therefore $ x := \hat x+x_0 \in \ri(R)$ and $A x = A\hat x+Ax_0 = (\hat s -z_0) + (z_0+s_0) = \hat s+ s_0 \in \ri(S).$

Next we show ``$\Leftarrow$''.  Suppose $ x\in \ri(R)$ is such that $ s:=A x \in \ri(S)$.  Then for any $\bar x\in R-R$ and $\bar s \in S-S$ both $\bar x + t  x\in R$ and $\bar s + t  s\in S$ provided $t > 0$ is large enough, and so $A\bar x - \bar s = A(\bar x + t  x) - (\bar s+t s) \in A(R)-S$.  Since this holds for any $\bar x\in R-R$ and $\bar s \in S-S$, it follows that
$ A(R-R) - (S-S) = A(R) - S$. Thus $A(R)-S$ is a linear subspace because $A(R-R) - (S-S)$ is evidently a linear subspace.
\end{proof}

It is easy to see that the primal and dual Slater conditions in parts (a) and (b) of Corollary~\ref{cor.onesided.dual} cannot both hold simultaneously.    Therefore, there is an inherent dual asymmetry between the Hoffman constants of the  feasibility problems~\eqref{eq.onesided.systems.general} and~\eqref{eq.onesided.systems.general.dual}.  Interestingly, in Section~\ref{sec.box} we show that there is a perfect symmetry between the Hoffman constants of the box-constrained feasibility problems~\eqref{eq.box-constrained} and~\eqref{eq.box-constrained.dual}.

The two examples below show that the right-hand side of each of the inequalities in Corollary~\ref{cor.onesided.dual} can be arbitrarily larger than the left-hand side.  For ease of computation, in the following examples
we let $S = \R^m$ and $R = \R^n_+$, and assume that $\R^m$ and $\R^n$ are endowed with the $\ell_2$-norm and the $\ell_1$-norm respectively and thus the dual spaces 
$(\R^m)^*$ and $(\R^n)^*$ are endowed with the $\ell_2$-norm and the $\ell_\infty$-norm respectively.

\begin{example}\label{ex.primal} Suppose $m=2, n = 4, \theta \in (0,\pi/6)$, 
\[
A = \matr{0 & 0 & \cos(\theta) & -1 \\ 
1&-1&-\sin(\theta)&0},
\]
and
$
\bar x := \matr{1+\sin(\theta)&1 & 1& \cos(\theta)}\transp.
$
It is easy to check that $A\bar x=0$,  and $\bar x>0$. Hence the primal Slater condition in Corollary~\ref{cor.onesided.dual}\ref{it:itema}
holds at $\bar x$. 
  
Thus identity~\eqref{eq.hoffman.special} in Theorem~\ref{thm.Hoffman.char} implies that 
$\H(P_A) = \|P_A\|$.  In addition, the choice of norms and identity~\eqref{eq.PA.char} in Example~\ref{ex.char} imply that $1/\|P_A\|$ is the Euclidean distance from the origin to the boundary of the convex hull of the columns of $A$.  This is the distance from the origin to the middle point of the segment joining $(0,1)$ and $(\cos(\theta),-\sin(\theta))$, that is,
\[
\sin\left(\frac{\pi-(\pi/2+\theta)}{2}\right) = \sin\left(\frac{\pi}{4}-\frac{\theta}{2}\right) \ge \sin\left(\frac{\pi}{6}\right).
\]
Thus we have $\H(P_A) = \|P_A\| \le 1/\sin(\pi/6)$.
On the other hand, $\H(P_A^*)\ge 1/\sin(\theta)$ because for $c = \matr{1/\sin(\theta)&0&-1&0}\transp$ and $v = \matr{0&0}\transp$ we have $\dist^*(c-A\transp v, \R^n_+) =\|(A\transp v - c)_+\|_\infty = 1$ and 
$y\in P_A^*(c) \Rightarrow y = \matr{0 & 1/\sin(\theta)}\transp.$ Thus
\[
y\in P_A^*(c) \Rightarrow \|v-y\|_2  = \|y\|_2 = y_2 \ge \frac{1}{\sin(\theta)}
\]
and so
\[
\dist(v,P_A^*(c)) \ge \frac{1}{\sin(\theta)} = \frac{1}{\sin(\theta)} \cdot \dist^*(c-A\transp v, \R^n_+).
\]
Since this holds for any $\theta \in (0,\pi/6)$, it follows that $\H(P_A^*)$ can be arbitrarily larger than $\H(P_A)$ by choosing $\theta$ sufficiently small.
\end{example}

\begin{example}\label{ex.dual} Suppose $m=2, n = 3, \phi \in (0,\pi/6)$, 
\[
A = \matr{\sin(\phi)&0 & 1 \\ 
1&1&0},
\]
and $\bar y = -\matr{1& 1}\transp$. It is easy to check that $A\transp \bar y<0$. Hence the dual Slater condition in Corollary~\ref{cor.onesided.dual}\ref{it:itemb}  holds at $\bar y$.

Thus 
identity~\eqref{eq.hoffman.special} in Theorem~\ref{thm.Hoffman.char} implies that 
\[
\H(P_A^*) = 
 \|P_A^*\|^* = \max_{\|z\|_\infty \le 1} \min_{A\transp y \le z} \|y\|_2.
\]  
The choice of norms and $A$ imply that the solution to this max-min problem  is attained at $z = -\matr{1 & 1 & 1}\transp $ and $y = -\matr{1 & 1}\transp$. Hence
\[
 \|P_A^*\|^* = \left\|-\matr{1\\1} \right\|_2 = \sqrt{2}.
\]
Thus we have $\H(P_A^*) = \sqrt{2}.$ On the other hand, $\H(P_A)\ge 2/\sin(\phi)$ because for $b = \matr{0&1}\transp$ and $u = \matr{1&0&0}\transp$ we have $\|Au-b\|_2 = \sin(\phi)$ and $x\in P_A(b)\Rightarrow x = \matr{0&1&0}\transp$.  Thus
\[
x\in P_A(b) \Rightarrow \|x-u\|_1 = 2
\]
and so
\[
\dist(u,P_A(b))= 2 = \frac{2}{\sin(\phi)} \cdot \|Au-b\|_2.
\]
Since this holds for any $\phi \in (0,\pi/6)$, it follows that $\H(P_A)$ can be arbitrarily larger than $\H(P_A^*)$ by choosing $\phi$ sufficiently small.
\end{example}

The following example, which is a concatenation of the above two examples, illustrates the crucial role the Slater condition plays in the duality inequality in Corollary~\ref{cor.onesided.dual}.  Example~\ref{ex.both} illustrates that when neither the primal nor dual Slater conditions hold, it can be quite challenging to determine the direction of the inequality between the Hoffman constants $\H(P_A)$ and $\H(P_A^*)$.  

\begin{example}\label{ex.both}
Suppose $m=4, n = 7, \theta\in (0,\pi/6) ,\phi \in (0,\pi/6)$ and 
\[
A = \matr{0 & 0 & \cos(\theta) & -1&0&0&0 \\ 
1&-1&-\sin(\theta)&0&0&0&0\\
0&0&0&0&\sin(\phi)&0 & 1 \\ 
0&0&0&0&1&1&0}.
\]
In this case the primal Slater condition $Ax=0, x>0$ does not hold for any $x\in \R^n$ because of the last row of $A$. The dual Slater condition $A\transp y < 0$ does not hold for any $y\in \R^m$ either because of the first two columns of $A$.  Let $A^1$ and $A^{2}$ denote respectively the upper-left $2 \times 4$ and lower-right $2\times 3$ submatrices of $A$.  The choice of norms in $\R^n$ and $\R^m$, the triangle inequality, and the block-diagonal structure of $A$ imply both
\[
\max\{\H(P_{A^1}),\H(P_{A^{2}})\} \le \H(P_{A}) \le \H(P_{A^1}) + \H(P_{A^{2}}),
\]
and
\[
\max\{\H(P_{A^1}^*),\H(P_{A^{2}}^*)\} \le \H(P_{A}^*) \le \H(P_{A^1}^*) + \H(P_{A^{2}}^*).
\]
Thus by proceeding as in Example~\ref{ex.primal} and Example~\ref{ex.dual},  
it follows that $\H(P_A) \ll \H(P_A^*)$ when $0 < \theta \ll \phi < \pi/6$ and $\H(P_A^*) \ll \H(P_A)$ when $0<\phi \ll \theta< \pi/6$.  

We should note that in this example, the task of determining the direction of the inequality between $\H(P_A)$ and $\H(P_A^*)$ is possible thanks to its particularly simple block-diagonal structure. This task would generally be far more challenging when neither the primal nor dual Slater conditions hold.
\end{example} 

\section{Box-constrained feasibility problems}
\label{sec.box}
This section details an interesting application of 
Theorem~\ref{thm.Hoffman.char} and Theorem~\ref{thm.dual.ineq} to a non-conic context.  We show that the Hoffman constants of the box-constrained feasibility problems~\eqref{eq.box-constrained} and~\eqref{eq.box-constrained.dual} are identical.  We also establish a striking identity between these Hoffman constants and the chi condition measures~\cite{Olea90,Stew89} for suitable choices of $A$ and $L$. 

\begin{theorem}\label{thm.box.dual} 
Suppose $A\in \R^{m\times n}$, $L\subseteq \R^n$ is a linear subspace, and $\B:=\{x\in \R^n: \ell \le x \le u\}$ for some $\ell,u\in \R^n$ with $\ell < u.$ Let $P_{A,\B,L}$ and $P_{A\transp,L^\perp,-\B}$ be the solution mappings to the box-constrained 
systems of linear inequalities~\eqref{eq.box-constrained} and~\eqref{eq.box-constrained.dual} respectively.  Then $\H(P_{A,\B,L}) = \H(P_{A\transp,L^\perp,-\B})$.
\end{theorem}
\begin{proof}

To ease notation,  let $\Phi:=P_{A,\B,L}$ and $\Psi:=P_{A\transp,L^\perp,-\B}$ throughout this proof. Before proving the identity $\H(\Phi) = \H(\Psi)$, we establish the convenient expressions \eqref{eq.HP} and~\eqref{eq.HQ} below for $\H(\Phi)$ and $\H(\Psi)$.  

Observe that $\graph(\Phi) = \{(Ax-s,x): x\in \B, s\in L\}$. Thus it follows that every $T\in \T(\Phi)$ is of the form $T = T_{I,J} = \{(Ax-s,x): x_I \ge 0, x_J \le 0, s\in L\}$ for some $I,J\subseteq [n]$ with $I\cap J = \emptyset$.   For simplicity, we shall write 
$\Phi_{I,J}$ as shorthand for $\Phi_{T_{I,J}}$ and
$\Phi_I$ as shorthand for $\Phi_{T_{I,I^c}}$.

We claim that the collection $\mathcal S:=\{T_{I,I^c}: I\subseteq[n]\}$ satisfies the bounding property and thus Theorem~\ref{thm.Hoffman.char} implies that 
\begin{equation}\label{eq.HP}
\H(\Phi)  =  \max_{I\subseteq[n]} \|\Phi_I\| = \max_{I\subseteq[n]} \;\H(\Phi_I).
\end{equation}
We next prove the above claim that $\mathcal S= \{T_{I,I^c}: I\subseteq[n]\}$ satisfies the bounding property.  To that end, suppose $I,J\subseteq [n]$ with $I\cap J = \emptyset$ and let $b\in \dom(\Phi_{I,J})$ be such that $\|b\|=1$ and 
\[
\|\Phi_{I,J}\|= \min_{x\in \Phi_{I,J}(b)} \|x\| = \min\{\|x\|: Ax-b\in L, x_I\ge 0, x_J\le 0\}.
\]
Suppose $\bar x$ is a minimizer of the last expression and let
\[
I':=I \cup\{i: \bar x_i > 0\}.
\]
This choice of $I'$ guarantees that $b\in \dom(\Phi_{I'})$ because 
$(b,\bar x) \in T_{I',(I')^c}$. The choice of $I'$ also guarantees
that $I\subseteq I'$ and $J\subseteq (I')^c$ because if $j\in J$ then $j\not \in I$ and $\bar x_j \le 0$ so $j\not \in I'$.
Therefore $T_{I',(I')^c} \subseteq T_{I,J}$ and 
\[
\|\Phi_{I'}\|\ge \min_{x\in \Phi_{I'}(b)} \|x\| \ge \min_{x\in \Phi_{I,J}(b)} \|x\| = \|\Phi_{I,J}\|.
\]
Since this holds for all $I,J\subseteq [n]$ with $I\cap J = \emptyset$, it follows that $\mathcal S=\{T_{I,I^c}: I\subseteq[n]\}$ satisfies the bounding property.

We can proceed in a similar fashion for $\Psi$.  Observe that $\graph(\Psi) = \{(A\transp y + s,y): y\in L^\perp, s\in \B\}$ and thus it follows that every $K\in \T(\Psi)$ is of the form $K = K_{I,J} = \{(A\transp y + s,y): y\in L^\perp, s_I \ge 0, s_J \le 0\}$ for some $I,J\subseteq [n]$ with $I\cap J = \emptyset$.   For simplicity, we shall write 
$\Psi_I$ as shorthand for $\Psi_{K_{I,I^c}}$.

Again the collection $\tilde{\mathcal S}:=\{K_{I,I^c}: I\subseteq[n]\}$ satisfies the bounding property and thus Theorem~\ref{thm.Hoffman.char} implies that 
\begin{equation}\label{eq.HQ}
\H(\Psi) =  \max_{I\subseteq[n] } \|\Psi_I\|^* = \max_{I\subseteq[n] } \;  \H(\Psi_I).
\end{equation}
Next, observe that for each $I \subseteq [n]$ the mappings $\Phi_I$ and $\Psi_I$ can be defined via
\begin{align*}
b&\mapsto \Phi_I(b) = \{x\in \R^n: Ax - b\in L, x_I \ge 0, x_{I^c}\le 0\} 
\\
c&\mapsto \Psi_I(c) = \{y\in \R^m: A\transp y + s = c, y\in L^\perp, s_I \ge 0, s_{I^c}\le 0\}. 
\end{align*}
Thus by applying Example~\ref{prop.adjoint} with $S=L$ and $R=\{x\in\R^n: x_I\ge 0, \, x_{I^c}\le 0\}$, it follows that $\Phi_I^* = \Psi_I$ for all 
$I\subseteq[n]$. Theorem~\ref{thm.dual.ineq} in turn implies that
\[
\|\Phi_I\|  \le \H(\Psi_I) \text{ and } \|\Psi_I\|^*  \le \H(\Phi_I)
\]
for all $I\subseteq[n]$.  The latter inequalities together with~\eqref{eq.HP} and~\eqref{eq.HQ} imply that $\H(\Phi) = \H(\Psi)$.
\end{proof}

We conclude this section by establishing a striking connection between the Hoffman constants of box-constrained feasibility problems and the chi condition measures introduced in the seminal papers~\cite{Olea90,Stew89} that we next recall.
 Suppose $A\in \R^{m\times n}$ is full row-rank,  $\R^m$ and $\R^n$ are endowed with Euclidean norms, and  $\mathscr D \subseteq \R^{n\times n}$ denotes the set of diagonal matrices in $\R^{n\times n}$ with positive diagonal entries.  The condition measures $\chi(A)$ and $\overline\chi(A)$ are defined as follows (see~\cite{Olea90,Stew89}):
\[
\chi(A):=\max_{D\in \mathscr D} \|(ADA\transp)^{-1}AD\|_2
\]
and
\[
\overline\chi(A):=\max_{D\in \mathscr D} \|A\transp(ADA\transp)^{-1}AD\|_2.
\]
We write $\|\cdot\|_2$ in the above expressions for $\chi(A)$ and $\overline \chi(A)$ to highlight that the operator norms in both cases are those induced by the Euclidean norms in $\R^n$ and $\R^m$.
Although it is not immediately evident, the constants $\chi(A)$ and $\overline \chi(A)$ are finite for any full row-rank matrix $A\in
\R^{m\times n}$.  This fact was independently shown by Ben-Tal and Teboulle~\cite{BenTT90},
Dikin~\cite{Diki74}, Stewart~\cite{Stew89}, and Todd~\cite{Todd90}.  The constants $\chi(A)$ and $\overline \chi(A)$
arise in and
play a key role in weighted least-squares problems~\cite{BobrV01,Fors96,ForsS01} and in linear
programming~\cite{HoT002,ToddTY01,Tunc99,VavaY96}.

We note that the quantity $\overline\chi(A)$ only depends on the space $\Image(A\transp)$ but not on the specific matrix $A\in\R^{m\times n}$.  In other words, $\overline\chi(A) = \overline\chi(\tilde A)$ for any other full row-rank matrix $\tilde A\in \R^{m\times n}$ such that $\Image(A\transp)=\Image(\tilde A\transp)$, or equivalently
$\overline\chi(A) = \overline\chi(RA)$ for any non-singular $R\in \R^{m\times m}$.  By contrast, $\chi(A)$ does depend on the specific matrix $A\in\R^{m\times n}$. Furthermore, since $\R^n$ and $\R^m$ are endowed with Euclidean norms, when the rows of $A$ are orthonormal we have $\|A\transp y\|_2 = \|y\|_2$ for all $y\in \R^m$.  Thus when the rows of $A$ are orthonormal, the following identity holds
\begin{equation}\label{eq.chi.barchi}
\overline\chi(A)=\chi(A).
\end{equation}

\medskip
Theorem~\ref{thm.chi.hoffman} below formalizes an interesting and surprising connection between Hoffman constants of  box-constrained feasibility problems and the chi condition measures $\chi(A)$ and $\overline \chi(A)$.  As we explain below, this result is  related to and enhances some equivalences between Hoffman constants, the chi condition measures, and other condition measures previously developed in our unpublished manuscript~\cite{PenaVZ19}.

\begin{theorem}\label{thm.chi.hoffman} Suppose  $A\in \R^{m\times n}$ is full row-rank and  both $\R^n$ and $\R^m$ are endowed with Euclidean norms.  Then for any $\ell,u\in \R^n$ with $\ell < u$ and $\B = \{x\in \R^n: \ell \le x \le u\}$ the following identities hold.  First, the chi condition measure $\chi(A)$ is identical to the Hoffman constants of the following box-constrained feasibility problems
\[
\begin{array}{r}
Ax-b = 0\\
x \in \B
\end{array}
\qquad 
\text{ and }
\qquad
\begin{array}{r}
c-A\transp y \in \B\;\;\;\\
y \in \R^m.
\end{array}
\]
In other words, 
\begin{equation}\label{eq.chi}
\H(P_{A,\B,\{0\}}) = \H(P_{A\transp,\R^m,-\B}) = \chi(A).
\end{equation}
Second, for $L_A:=\ker(A)\subseteq \R^n$ the chi condition measure $\overline\chi(A)$ is identical to the Hoffman constants of the following box-constrained feasibility problems
\[
\begin{array}{r}
x-b\in L_A\\
x \in \B
\end{array}
\qquad 
\text{ and }
\qquad
\begin{array}{r}
c- y \in \B\;\;\;\\
y \in L_A^\perp.
\end{array}
\]
In other words, 
\begin{equation}\label{eq.bar.chi}
\H(P_{I_n,\B,L_A}) = \H(P_{I_n,L_A^\perp,-\B}) =\overline\chi(A),
\end{equation}
where $I_n$ denotes the $n\times n$ identity matrix.
\end{theorem}
The proof of Theorem~\ref{thm.chi.hoffman} relies on Lemma~\ref{cor.signed} which is of independent interest.  The statement of Lemma~\ref{cor.signed} relies on the following set of {\em signature matrices} $\SS \subseteq\R^{n\times n}$:
\begin{equation}\label{eq.signature}
\SS:=\{\Diag(d): d \in \{-1,1\}^n\}.
\end{equation}

\begin{lemma}\label{cor.signed}
Suppose $A\in \R^{m\times n}$, $L\subseteq \R^n$ is a linear subspace, and $\B:=\{x\in \R^n: \ell \le x \le u\}$ for some $\ell,u\in \R^n$ with $\ell < u.$ In addition, suppose $\R^n$ and $\R^m$ are endowed with Euclidean norms. 
Then
\begin{equation}\label{eq.signed.identity}
\H( P_{A,\B,L}) = \max_{D\in \SS} \H( P_{AD,\R^n_+,L}) \; \text{ and } \; 
\H(P_{A\transp,L^\perp,-\B})=\max_{D\in \SS} \H(P_{DA\transp,L^\perp,-\R^n_+}).
\end{equation}
In particular, 
\begin{equation}\label{eq.signed.identity.special}
\H( P_{A,\B,\{0\}}) = \max_{D\in \SS} \H( P_{AD})  \; \text{ and } \;  
 \H(P_{A\transp,\R^m,-\B})=\max_{D\in \SS} \H(P_{AD}^*),
\end{equation}
where $P_A$ and $P_A^*$ are as
defined in~\eqref{eq.PA} and~\eqref{eq.QA} respectively.

\end{lemma}
\begin{proof}
There is a natural one-to-one correspondence between the subsets of $[n]$ and the set of signature matrices $\SS$, namely
\[
I\subseteq [n] \leftrightarrow D_I = \Diag(d^I) \in \SS, \text{ where } d^I_i = 1 \text{ for } i\in I \text{ and } d^I_i = -1 \text{ for } i\not \in I.
\]
This correspondence and the fact that $\R^n$ is endowed with the Euclidean norm allow us to restate the  identities~\eqref{eq.HP},~\eqref{eq.HQ} in the proof of Theorem~\ref{thm.box.dual} as follows
\[
\H( P_{A,\B,L}) = \max_{D\in \SS} \H( P_{AD,\R^n_+,L})\]
and
\[
\H(P_{A\transp,L^\perp,-\B}) = \max_{D\in \SS} \H(P_{DA\transp,L^\perp,-\R^n_+}).
\]
Therefore~\eqref{eq.signed.identity} follows.  Identity ~\eqref{eq.signed.identity.special} in turn follows from~\eqref{eq.signed.identity} and the definition~\eqref{eq.PA} and~\eqref{eq.QA} of $P_A$ and $P_A^*$ respectively.
\end{proof}

The proof of Theorem~\ref{thm.chi.hoffman} relies on the following notation.  For $A\in \R^{m\times n}$ and $I\subseteq [n]$ let $A_I$ denote the $m\times |I|$ submatrix of $A$ defined by the columns in $I$.

\begin{proof}[Proof of Theorem~\ref{thm.chi.hoffman}] The first identity~\eqref{eq.chi} is an immediate consequence of Theorem~\ref{thm.box.dual}, identity~\eqref{eq.signed.identity.special} in Lemma~\ref{cor.signed} and the following two facts.

\noindent
{\bf Fact 1.} The quantity $\chi(A)$ can be characterized as follows~\cite[Corollary 2.2]{Fors96}:  
\begin{equation}\label{fact.1}
\chi(A) = \max_{I\subseteq [n], |I|=m \atop A_I \text{nonsingular}} \max_{x\in \R^I \atop \|A_Ix\|_2 \le 1} \|x\|_2.
\end{equation}
{\bf Fact 2.} The quantity $\H(P_A^*)$ can be characterized as follows (see, e.g.~\cite[Theorem 2.7]{KlatT95} or~\cite[Lemma 15]{WangL14}):
\begin{equation}\label{fact.2}
\H(P_A^*) = \max_{I\subseteq [n], |I|=m \atop A_I \text{nonsingular}} \max_{x\in \R^I_+ \atop \|A_Ix\|_2 \le 1} \|x\|_2.
\end{equation}
Indeed,~\eqref{fact.1}~and~\eqref{fact.2} imply that 
\begin{equation}\label{eq.chi.hoffman}
\max_{D\in \SS} \H(P_{AD}^*) \le \max_{D\in \SS} \chi(AD) = \chi(A).
\end{equation}
On the other hand, suppose $\bar I$ and $\bar x\in \R^{\bar I}$ attain $\chi(A)$ in~\eqref{fact.1}.  Choose $\hat D\in \SS$ such that $\hat D_{ii} = \text{sign}(\bar x_{i})$ for each $i\in \bar I$ and let $\hat x:= \hat D_I \bar x \in \R^{\bar I}_+$.  Observe that $(A\hat D)_{\bar I} = A_{\bar I}\hat D_{\bar I}$ is nonsingular and 
\[
\|(A\hat D)_{\bar I} \hat x\|_2 = \|A_{\bar I}\hat D_{\bar I}\hat x\|_2 = \|A_{\bar I}\bar x\|_2 = 1.
\]
Thus by~\eqref{fact.2}
\begin{equation}\label{eq.chi.hoffman.2}
\max_{D\in \SS} \H(P_{AD}^*) \ge \H(P_{A\hat D}^*) \ge \max_{x\in \R^{\bar I}_+ \atop \|(A\hat D)_{\bar I} x\|_2 \le 1} \|x\|_2 \ge \|\hat x\|_2 = \|\bar x\|_2= \chi(A).
\end{equation}
Identity~\eqref{eq.chi} thus follows by putting together~\eqref{eq.chi.hoffman},~\eqref{eq.chi.hoffman.2},~\eqref{eq.signed.identity.special}, and Theorem~\ref{thm.box.dual}. 

To prove the second identity~\eqref{eq.bar.chi}, let $\tilde A \in \R^{m \times n}$ be a matrix with orthonormal rows such that $\ker(\tilde A) = L_A$. Then by~\eqref{eq.chi} 
and Theorem~\ref{thm.box.dual}
\begin{equation}\label{eq.bar.chi.equiv}
\overline\chi(A) = \chi(\tilde A) = \H(P_{\tilde A\transp,\R^m,-\B}).
\end{equation}
Since the rows of  $\tilde A$ are orthonormal, it holds that $\|\tilde A\transp y\|_2 = \|y\|_2$ for all $y\in\R^m$.  Hence $ \H(P_{\tilde A\transp,\R^m,-\B}) = \H(P_{I_n,L_A^\perp,-\B})$ and thus~\eqref{eq.bar.chi} follows from~\eqref{eq.bar.chi.equiv}.   
\end{proof}

To conclude this section we briefly comment on the relationship between Theorem~\ref{thm.chi.hoffman} and some previous developments in the unpublished manuscript~\cite{PenaVZ19}.  In~\cite[Theorem 1]{PenaVZ19} we showed that
\[
\chi(A) = \max_{D\in \SS} \H(P_{AD}^*).
\]
We replicate this result as part of identity~\eqref{eq.signed.identity.special} in 
Lemma~\ref{cor.signed}.  Identity~\eqref{eq.signed.identity.special} in turn is a key step in the proof of Theorem~\ref{thm.chi.hoffman}.  Theorem~\ref{thm.chi.hoffman} adds the interesting novel extension that these quantities match the box-constrained Hoffman constants $\H(P_{A,\B,\{0\}}) = \H(P_{A\transp,\R^m,-\B})$.  

The manuscript~\cite{PenaVZ19} also presents other interesting results relating Hoffman constants, chi numbers, Renegar's and Grassmannian condition numbers.

\section{Conclusions}
Theorem~\ref{thm.dual.ineq} provides a duality inequality between the norm of a polyhedral sublinear mapping and the Hoffman constant of its upper adjoint.  As a consequence, we obtain a novel duality relationship between the Hoffman constants of the feasibility problems~\eqref{eq.onesided.systems.general} and~\eqref{eq.onesided.systems.general.dual} whenever $R,S$ are polyhedral cones and either a primal or a dual Slater condition holds.  Theorem~\ref{thm.dual.ineq} also yields an interesting identity between the Hoffman constants of the box-constrained feasibility problems~\eqref{eq.box-constrained} and~\eqref{eq.box-constrained.dual}.  We conjecture that similar identities hold for other non-conic pairs of polyhedral feasibility problems.

For the special case when the underlying spaces are endowed with Euclidean norms, we also establish a surprising identity between the Hoffman constants of the box-constrained feasibility problems~\eqref{eq.box-constrained} and~\eqref{eq.box-constrained.dual} and the chi condition measures~from~\cite{Olea90,Stew89}. We conjecture that this identity can be extended to the kappa measure of~\cite{DaduHNV20} and could be used to shed new light on the various {\em proximity results} that underlie important developments for linear programming such as those documented in~\cite{ApplHLL22,DaduHNV24,MontT03,VavaY96}.  We will pursue this line of research in some future work.

The duality results developed in this paper also set the stage for more general duality developments.  In particular,  the recent paper~\cite{CamaCGP22} extends our characterizations of the Hoffman constant~\cite{PenaVZ21} to a more general context concerning the {\em calmness modulus} of solution mappings for semi-infinite systems of linear inequalities.  It is thus natural to explore what duality relationships hold in that more general context. The calmness modulus can be seen as a {\em local} version of the (global) Hoffman constant defined in~\eqref{eq.def.Hoffman}.  More precisely, as described in~\cite[Section 3.8]{DontR09}, a set-valued mapping $\Phi:\R^p\rightrightarrows \R^q$ is calm at $\bar u$ for $\bar v$ if $(\bar u,\bar v)\in \graph(\Phi)$ and there exists a constant $\kappa \ge 0$ along with neighborhoods $U$ of $\bar u$ and $V$ of $\bar v$ such that
\[
\dist(v,\Phi(\bar u))\le \kappa \cdot \dist(\bar u,\Phi^{-1}(v)\cap U) \text{ for all } v\in V.
\] 
Notice the similarity between this inequality and the inequality preceding~\eqref{eq.def.Hoffman}.

The calmness modulus of $\Phi$ at $\bar u$ for $\bar v$ is the infimum of $\kappa$ over all neighborhoods $U$ of $\bar u$ and $V$ of $\bar v$ such that the above inequality holds.  A duality result for calmness would require extensions of, or new techniques beyond, the global and polyhedral nature of our main foundational blocks, namely Theorem~\ref{thm.Hoffman.char} and Theorem~\ref{thm.dual.ineq}.

The recent work of Wirth et al.~\cite{WirtPP25} on the Frank-Wolfe algorithm relies heavily on the {\em inner facial distance} and {\em outer facial distance} of a polytope, both of which have natural geometric interpretations and local versions as discussed in~\cite{PenaR19,WirtPP25}.  The inner facial distance was initially introduced in~\cite{PenaR19} and shown to be a Hoffman constant in~\cite{GutmP21}. 
We conjecture that the outer facial distance is a Hoffman constant as well but that has not been formally shown as of the time of this writing.  We also conjecture that there is a duality relationship between the inner and outer facial distances.  Positive answers to these conjectures would in turn have interesting implications on the convergence properties of first-order algorithms in light of the recent developments in~\cite{ApplHLL22,GutmP21,Hind23,LacoJ15,PenaR19,WirtPP25}.


\end{document}